\documentclass[reqno,12pt]{amsart}
\newcommand{\CC}{{\mathbb C}}

\def\bege{\begin{equation}} \def\ende{\end{equation}}
\def\begr{\begin{eqnarray}} \def\endr{\end{eqnarray}}
\usepackage{amsmath}
\usepackage{amssymb}
\usepackage{cite}

\usepackage{bbm}
\usepackage{amsmath}
\usepackage{txfonts}
\usepackage{stmaryrd}
\usepackage{amssymb}
\usepackage{amsfonts}
\usepackage{times}
\usepackage{mathrsfs}
\usepackage{tikz}

\usepackage{amsthm}
\usepackage{enumerate}

\usepackage{framed}
\usepackage{lipsum}
\usepackage{color}

\def\CC{ \mathbb{C}}
\newcommand{\DD}{{\mathbb D}}

\def\T{\mathbb{T}}

\def\D{\mathbb{D}}

\def\a{\alpha}
\def\b{\beta}

\def\begr{\begin{eqnarray}} \def\endr{\end{eqnarray}}

\def\gaon{ \Gamma(\eta)\cap\Omega_{\epsilon}(f)}

\def\on{\Omega_{\epsilon}(f)}
\def\a{\alpha}\def\b{\beta}
\def\msk{\medskip}
\def\ol{\overline}

\def\f{\frac}
\def\dt{\mathcal{C}_{\mathcal{B}_{{\log}^\gamma}^\beta}(DT_p^q(\alpha)\cap\mathcal{B}_{{\log}^\gamma}^\beta)}
\def\l{\log\frac{e}{1-|z|^2}}
\def\Blog{\mathcal{B}_{{\log}^\gamma}^\beta}
\def\Blogo{   \mathcal{B}_{{\log}^\gamma,0}^\beta}

\def\msk{\medskip}
\def\ol{\overline}
\newtheorem{Lemma}{Lemma}[section]
\newtheorem{Theorem}[Lemma]{Theorem}
\newtheorem{Corollary}[Lemma]{Corollary}

\newcounter{other}            

\usepackage{pgfplots}

\begin{document}
	\title[]{The closure of derivative tent spaces in the logarithmic Bloch-type norm}
	
	\author{Rong Yang and Xiangling Zhu$\dagger$}
	
	\address{Rong Yang
	\\ Institute of Fundamental and Frontier Sciences, University of Electronic Science and Technology of China, 610054, Chengdu, Sichuan, P.R. China.}
	\email{yangrong071428@163.com  }

\address{Xiangling Zhu
	\\ University of Electronic Science and Technology of China, Zhongshan Institute, 528402, Zhongshan, Guangdong, P.R.
	China.}
    \email{jyuzxl@163.com}


	\begin{abstract}

In this paper,  the derivative tent space \(DT_p^q(\alpha)\) is introduced.  
Then, we study \(\mathcal{C}_{\mathcal{B}_{{\log}^\gamma}^\beta}(DT_p^q(\alpha)\cap\mathcal{B}_{{\log}^\gamma}^\beta)\), the  closure of the derivative tent space \(DT_p^q(\alpha)\) in the logarithmic Bloch-type space \(\Blog\).	As a byproduct, some new characterizations for \(C_\mathcal{B}(\mathcal{D}^p_{\alpha} \cap \mathcal{B})\) and \(C_{\mathcal{B}_{{\log}}}(\mathcal{D}^2_{\alpha}\cap\mathcal{B}_{{\log}})\) are obtained.

	\thanks{$\dagger$ Corresponding author.}
	\end{abstract}
	
	\maketitle

\section{Introduction}

Let $\mathbb{D}$ be the open unit disk in the complex plane $\mathbb{C}$.  Define $H(\mathbb{D})$ as the set of all analytic functions on $\mathbb{D}$. 
Let $\zeta > \frac{1}{2}$ and $\eta \in \mathbb{T}$, the boundary of $\DD$. The non-tangential approach region $\Gamma_\zeta(\eta)$ is defined as
\[
\Gamma(\eta)=\Gamma_\zeta(\eta) = \left\{z \in \mathbb{D} : |z - \eta| < \zeta(1 - |z|^2)\right\}.
\]

For $0 < p, q < \infty$ and $\alpha > -2$, the tent space $T_p^q(\alpha)$ consists of all measurable functions $f$ on $\mathbb{D}$ such that
\begin{align*}
	\|f\|_{T_p^q(\alpha)}^q=\int_\mathbb{T}\left(\int_{\Gamma(\eta)}|f(z)|^p(1-|z|^2)^{\a} d A(z)\right)^{\frac{q}{p}}|d \eta|<\infty.
\end{align*}
Here \(dA(z) = \frac{1}{\pi}dx dy\) is the normalized Lebesgue area measure on \(\mathbb{D}\). Tent spaces were first presented in the work of Coifman, Meyer, and Stein \cite{cms} to tackle problems in harmonic analysis. They created a unified framework for studying problems related to classical function spaces such as Hardy and Bergman spaces. In the above definition, the aperture $\zeta$ of the non-tangential region $\Gamma_\zeta(\eta)$ isn't explicitly stressed. This is because for any two different apertures, the resulting function spaces have equivalent quasi-norms. 

Denote the intersection of \(T_p^q(\a)\) and \(H(\D)\) as \(AT_p^q(\a)\) (the analytic tent space). When \(q = p\), \(AT_p^p(\alpha)=A^p_{\alpha + 1}\), where \(A^p_{\alpha+1}\) is the weighted Bergman space. Moreover, a function \(f\in H^q\) if and only if \(f'\in AT_2^q\). Here, \(H^q\) is the Hardy space. That is:
$$
 \int_{\mathbb{T}} \left( \int_{\Gamma(\eta)} |f'(z)|^2dA(z) \right)^{\frac{q}{2}} |d\eta| < \infty.
$$
This result is due to Marcinkiewicz and Zygmund \cite{ma} for \(p > 1\), and Calder\'on \cite{ca} extended it to the case \(0 < p \leq 1\). 
For the case of the unit ball, see \cite[Theorem 5.3]{pj}. For more information on the analytic tent space, see \cite{pa, yhl}.

 For $0< p<\infty$ and $\beta>-1$,  a  function $f \in H(\DD)$ belongs to the weighted Dirichlet  space $\mathcal{D}_{\beta}^p$  if
   $$\|f\|_{\mathcal{D}_{\beta}^p}^{p}=|f(0)|^{p}+\int_{\DD}|f'(z)|^p(1-|z|^2)^{\beta}dA(z)<\infty.$$ 
 The weighted Dirichlet space $\mathcal{D}_{p}^p$ is just the Bergman space $A^p$.  
   In particular, when $p=2$, the weighted Dirichlet space $\mathcal{D}_{1}^2$ is just the Hardy space $H^2$.

Let \(0 < p, q <\infty\) and \(\alpha >-2\).  Inspired by the definition of the tent space and the above mentioned results,  it is natural to define the derivative tent space  as follows:
\[
DT_p^q(\alpha)=\left\{ f\in H(\mathbb{D}): 	\|f\|_{DT_p^q(\alpha)} <\infty \right\},
\]
where
\begin{align*}
	\|f\|_{DT_p^q(\alpha)}^{q}=|f(0)|^q+\int_{\mathbb{T}}\left(\int_{\Gamma(\eta)}|f'(z)|^p(1 - |z|^2)^{\alpha}dA(z)\right)^{\frac{q}{p}}|d\eta|.
\end{align*} 
It is clear that the weighted Dirichlet space \(\mathcal{D}^p_{\beta}\) is just the derivative tent space \(DT_p^p(\beta-1)\) when $\beta>-1$.   A function $f\in H^q$  if and only if $f\in DT_2^q$.   A function $f\in A^p$  if and only if $f\in DT_p^p(p-1)$.
We believe this new space provides new perspectives for studying tent spaces and new ways to investigate Hardy spaces, weighted Bergman spaces and weighted Dirichlet spaces.

For   $\beta>0$ and  $\gamma\geq0$, let us recall the definition of the logarithmic Bloch-type space, denoted as $\Blog$. This space  consists of all $f\in H(\D)$  such that (\cite{st2})
$$
\|f\|_{\Blog}=|f(0)|+\sup_{z\in\D}(1 - |z|^{2})^{\beta}|f'(z)|\left( \l\right)^\gamma<\infty.
$$
Equipped with the norm $\|\cdot\|_{\Blog}$, $\Blog$ forms a Banach space.
When $\gamma = 0$, we get the Bloch-type space $\mathcal{B}^\beta$. Notably, for $\beta = 1$ and $\gamma = 0$, it's the classical Bloch space $\mathcal{B}$. When $\beta = 1$ and $\gamma = 1$, we have the logarithmic Bloch space $\mathcal{B}_{{\log}}$.    The little logarithmic Bloch-type space, denoted by $\Blogo$, is the set of all $f\in\Blog$ satisfying the condition
$$\lim_{|z|\to 1}(1 - |z|^{2})^{\beta}|f'(z)|\left( \l\right)^\gamma=0.$$
When $\beta = 1$ and $\gamma = 0$, the little logarithmic Bloch-type space $\Blogo$ is just the little Bloch space $\mathcal{B}_0$.
 
Let \(X\) be a subspace of \(Y\), \(C_Y(X)\) the closure of \(X\) in the \(Y\)-norm. In \cite{acp}, Anderson et al. raised an open question on the closure of \(H^\infty\) in the Bloch norm. Ghatage and Zheng described the closure of BMOA in the Bloch norm \cite{gz}. 
Monreal and Nicolau in \cite{mn} characterized \(C_\mathcal{B}(H^p \cap \mathcal{B})\) for \(1 < p < \infty\).
Galanopoulos, Monreal, and Pau extended this result to the range \(0 < p < \infty\) in \cite{gmp}.
In \cite{zhao1}, Zhao investigated the closure of certain Möbius invariant spaces in the Bloch norm.
Bao and G\"{o}\.{g}\"{u}\c{s} studied the closure of the spaces \(\mathcal{D}^2_{\alpha}(-1<\alpha\le 1)\) in the Bloch norm in \cite{bg}.
In \cite{gg}, Galanopoulos and Girela characterized \(C_\mathcal{B}(\mathcal{D}^p_{\alpha} \cap \mathcal{B})\) when \(1 \leq p < \infty\) and \(\alpha > -1\).
Qian and Li characterized \(C_{\mathcal{B}_{{\log}}}(\mathcal{D}^2_{\alpha}\cap\mathcal{B}_{{\log}})(\alpha>0)\) in \cite{ql}.
Subsequently, Bao, Lou, and Zhou addressed the open question  posed by Qian  and Li in \cite{blz} and investigated \(C_{\mathcal{B}_{{\log}}}(\mathcal{D}^2_{\alpha}\cap\mathcal{B}_{{\log}})(\alpha>-1)\). 
 For further research on closures, refer to  \cite{lr,chen,zy} and the references therein.

In this paper, we study  the closure of \(DT_p^q(\alpha)\cap\Blog\) in the logarithmic Bloch-type norm. As a by-product, we obtain some new characterizations for \(C_\mathcal{B}(\mathcal{D}^p_{\alpha} \cap \mathcal{B})\)  and \(C_{\mathcal{B}_{{\log}}}(\mathcal{D}^2_{\alpha}\cap\mathcal{B}_{{\log}})\).

Throughout this paper, we assert that $E \lesssim F$ if there exists a constant $C$ such that $E \leq CF$. The notation $E \asymp F$ signifies that both $E \lesssim F$ and $F \lesssim E$.\vskip 5mm

\section{Main results and proofs}
In this section, we describe the closure of the space $DT_p^q(\alpha) \cap \Blog$ in the logarithmic Bloch-type norm.  First, we state some lemmas. The following lemma is crucial   in our proof.

\begin{Lemma}\cite[Lemma 4]{ar}\label{le1}
	Let $0<p, q<\infty$ and $\lambda>\max \left\{1, \frac{p}{q}\right\}$. Then there are constants $C_1=C_1(p, q, \lambda)$ and $C_2=C_2(p, q, \lambda)$ such that
	$$
	C_1 \int_{\mathbb{T}} \mu(\Gamma(\eta))^{\frac{q}{p}}|d \eta| \leq \int_{\mathbb{T}}\left(\int_{\mathbb{D}}\left(\frac{1-|z|^2}{|1-z \bar{\eta}|}\right)^\lambda d\mu\right)^{\frac{q}{p}}|d \eta| \leq C_2 \int_{\mathbb{T}} \mu(\Gamma(\eta))^{\frac{q}{p}}|d \eta|
	$$
	for every positive measure $\mu$ on $\mathbb{D}$.
\end{Lemma}

The following three integral estimates are of great importance in our proof.

 \begin{Lemma}\cite[Lemma 2.5]{of}\label{le3}
	Let  $s>-1$, $r,t>0$ and $r+t-s-2>0$.   		If $r,t<s+2$, then
		$$
		\int_\D\frac{(1-|z|^2)^s}{|1-\bar{a}z|^{r}|1-\bar{b}z|^t}dA(z)\lesssim \frac{1}{|1-\bar{a}b|^{r+t-s-2}}
		$$
		for all $a,b\in \D$.
\end{Lemma}

\begin{Lemma}\cite[Proposition 2.4]{cz}\label{log}
	Let $p\ge0$, $s>-1$ and $c>0$. Then 
	$$
	\int_{\D}\f{(1-|z|^2)^{s}}{|1-\overline{z}w|^{2+s+c}\left(\log\frac{e}{1-|z|^2}\right)^p}dA(z) \lesssim \frac{1}{(1-|w|^2)^c\left(\log\frac{e}{1-|w|^2}\right)^p}
	$$
	for all $w\in \D$.
\end{Lemma}

\begin{Lemma}\cite[Lemma 3]{hl}\label{llog}
	Let $p\ge0$, $s>-1$, $r,t>0$ and $r+t-s-2>0$, $r<s+2<t$. Then 
	$$
	\int_{\D}\frac{(1-|z|^2)^{s}}{|1-\ol{a}z|^{r}|1-z\ol{w}|^{t}\left( \l \right)^p}dA(z) \lesssim \frac{1}{|1-\ol{a}w|^{r}(1-|w|^2)^{t-s-2}\left( \log\frac{e}{1-|w|^2} \right)^p} 
	$$
	for all $a,w \in \D$.
\end{Lemma}

 To clearly state and prove our main result, we  introduce a notation. Let $f\in H(\D)$ and $\epsilon>0$. Define the level set
$$
\Omega_{\epsilon}(f)=\left\{z\in \D: |f'(z)|(1-|z|^2)^\b\left(\log\frac{e}{1-|z|^2}\right)^\gamma\geq \epsilon\right\}.
$$
Now we are in a position to state and prove our main result.

\begin{Theorem}\label{th1}
Let $0<p,q,\b<\infty$, $\gamma\ge 0$, $\a>-2$. Then the following statements hold.
\noindent \begin{enumerate}
\item [(i)] If $\b<\frac{\a+2}{p}$, then $\mathcal{C}_{\Blog}(DT_p^q(\alpha)\cap\Blog)=\Blog$.
		
\item [(ii)] If $\b> \frac{\a+2}{p}+\frac{1}{q}$, then $\mathcal{C}_{\Blog}(DT_p^q(\alpha)\cap\Blog)=\Blogo$.
		
\item [(iii)]
If  $\frac{\a+2}{p}\le\b\le\frac{\a+2}{p}+\frac{1}{q}$, $1\le p<\a+3$ and $f\in\Blog$,  then $f\in\dt$ if and only if for any $\epsilon>0$,
\begin{align}\label{11}
	\int_{\T}\left(\int_{\gaon}\frac{(1-|z|^2)^{\a-p\beta}}{\left(\log\frac{e}{1-|z|^2}\right)^{p\gamma}} dA(z)
	\right)^{\frac{q}{p}}|d\eta|<\infty.
\end{align}
If  $\frac{\a+2}{p}\le\b\le\frac{\a+2}{p}+\frac{1}{q}$, $0<p<1$ and the function $|f'(z)|(1-|z|^2)^\b \left(\l\right)^\gamma$ is uniformly continuous with respect to the Bergman metric on $\D$, then $f\in\dt$ if and only if $(\ref{11})$ holds for any $\epsilon>0$.		
		
\end{enumerate}
\end{Theorem}

\begin{proof}
{\it (i)} Let $\b<\frac{\a+2}{p}$. For  $f\in\Blog$, we get
\begin{align*}
	\|f\|^q_{DT_p^q(\a)}\asymp&\int_{\T}\left(\int_{\Gamma(\eta)} |f'(z)|^p(1-|z|^2)^{\a} dA(z) \right)^\frac{q}{p}|d\eta|\\
	\lesssim&\|f\|^q_{\Blog}\int_{\T}\left(\int_{\Gamma(\eta)}\frac{(1-|z|^2)^{\a-p\b}}{\left(\log\frac{e}{1-|z|^2} \right)^{p\gamma}}dA(z) \right)^{\frac{q}{p}}|d\eta|\\
	\lesssim&\|f\|^q_{\Blog}\int_{\T}\left(\int_{\Gamma(\eta)}(1-|z|^2)^{\a-p\b}dA(z) \right)^{\frac{q}{p}}|d\eta|\\
	\lesssim &\|f\|^q_{\Blog}.
\end{align*}
This yields that $\Blog\subset DT_p^q(\a)$. Consequently,   $\mathcal{C}_{\Blog}(DT_p^q(\alpha)\cap\Blog)=\Blog$.

{\it (ii)} Let $\b> \frac{\a+2}{p}+\frac{1}{q}$.
By \cite[Lemma 2.6]{yhl},  
for $f\in DT_p^q(\a)$, we obtain
$$
|f'(z)|\lesssim\frac{\|f\|_{DT_p^q(\a)}}{(1-|z|^2)^{\frac{\a+2}{p}+\frac{1}{q}}},\quad z\in\D.
$$ 
 Since polynomials are  dense  in $DT_p^q(\a)$.
Therefore, $\lim_{|z|\to1}(1-|z|^2)^{\frac{\a+2}{p}+\frac{1}{q}}|f'(z)|=0.$  Hence,
$DT_p^q(\a)\subset \mathcal{B}_0^{\frac{\a+2}{p}+\frac{1}{q}}\subset\Blogo,$ which implies 
$$
\mathcal{C}_{\Blog}(DT_p^q(\alpha)\cap\Blog)=C_{\Blog}(DT_p^q(\a))\subset\Blogo.
$$
Since $\Blogo$ is the closure of polynomials in $\Blog$, it follows that $$\Blogo\subseteq C_{\Blog}(DT_p^q(\a)).$$
 Therefore, we conclude that $\mathcal{C}_{\Blog}(DT_p^q(\alpha)\cap\Blog)=\Blogo$.

{\it (iii)} {\bf Necessity.}
Suppose that $f\in \dt$. 
Then  for any $\epsilon > 0$, there exists a function $g \in DT_p^q(\alpha) \cap \Blog$ such that  $\|f - g\|_{\Blog} \leq \frac{\epsilon}{2}$. Observing that
	\begin{align*}
	& (1-|z|^2)^\b|f'(z)|\left(\l\right)^\gamma\\
	\leq& (1-|z|^2)^\b|g'(z)|\left(\l\right)^\gamma+ (1-|z|^2)^\b|f'(z)-g'(z)|\left(\l\right)^\gamma, \ \ z\in \D,
	\end{align*}
	we have $\Omega_{\epsilon}(f) \subseteq \Omega_{\frac{\epsilon}{2}}(g)$. 
	Hence,
	\begin{align*}
	\infty>&\int_{\T}\left( \int_{\Gamma(\eta)}|g'(z)|^p(1-|z|^2)^{\a} dA(z) \right)^{\frac{q}{p}}|d\eta|\\
	\ge& \int_{\T}\left( \int_{\Gamma(\eta)\cap\Omega_{\frac{\epsilon}{2}}(g)}|g'(z)|^p(1-|z|^2)^{\a} dA(z) \right)^{\frac{q}{p}}|d\eta|\\
	=&\int_{\T}\left( \int_{\Gamma(\eta)\cap\Omega_{\frac{\epsilon}{2}}(g)}|g'(z)|^p(1-|z|^2)^{p\beta}\left(\l\right)^{p\gamma}
	\frac{(1-|z|^2)^{\a-p\beta}}{\left(\log\frac{e}{1-|z|^2}\right)^{p\gamma}} dA(z)\right)^{\frac{q}{p}}|d\eta|\\
	\ge&\left(\frac{\epsilon}{2}\right)^q \int_{\T}\left( \int_{\Gamma(\eta)\cap\Omega_{\frac{\epsilon}{2}}(g) }\frac{(1-|z|^2)^{\a-p\beta}}{\left(\log\frac{e}{1-|z|^2}\right)^{p\gamma}} dA(z) \right)^{\frac{q}{p}}|d\eta|\\
	\ge& \left(\frac{\epsilon}{2}\right)^q \int_{\T}\left(\int_{\gaon}\frac{(1-|z|^2)^{\a-p\beta}}{\left(\log\frac{e}{1-|z|^2}\right)^{p\gamma}} dA(z) \right)^{\frac{q}{p}}|d\eta|,
	\end{align*}
    which implies the desired result.
	
	{\bf Sufficiency.}
	Suppose that $(\ref{11})$ holds. Without loss of generality, we may assume that $f(0)=0$.   
 Choose $\delta>0$ large enough.  For any $z\in \D$, by  \cite[Proposition 4.27]{z1} we obtain
	$$
	f(z)= \int_{\D}\f{f'(w)(1-|w|^2)^{1+\delta}}{\ol{w}(1-z\overline{w})^{2+\delta}}dA(w).
	$$
	Write $f(z)=f_1(z)+f_2(z)$, where
	$$
	f_1(z)= \int_{\Omega_{\epsilon}(f)}\f{f'(w)(1-|w|^2)^{1+\delta}}{\ol{w}(1-z\overline{w})^{2+\delta}}dA(w)
	$$
	and
	$$
	f_2(z)= \int_{\D\backslash\Omega_{\epsilon}(f)}\f{f'(w)(1-|w|^2)^{1+\delta}}{\ol{w}(1-z\overline{w})^{2+\delta}}dA(w).
	$$
	By calculation, we get
	$$
	f_1'(z)= (\delta+2) \int_{\Omega_{\epsilon}(f)}\f{f'(w)(1-|w|^2)^{1+\delta}}{(1-z\overline{w})^{3+\delta}}dA(w)
	$$
	and
	$$
	f_2'(z)= (\delta+2) \int_{\D\backslash\Omega_{\epsilon}(f)}\f{f'(w)(1-|w|^2)^{1+\delta}}{(1-z\overline{w})^{3+\delta}}dA(w).
	$$
	Let $g(z)=f_1(z)-f_1(0)$. Then $g(0)=0$. 
	Using Lemma \ref{log}, we have
	\begin{align*}
	&\|f-g\|_{\Blog}	\asymp\sup_{z\in\D}(1-|z|^2)^\b|f_2'(z)|\left(\l\right)^\gamma\\
	\lesssim& \sup_{z\in\D} (1-|z|^2)^\b\left(\l\right)^\gamma \int_{\D\backslash\Omega_{\epsilon}(f)}\f{|f'(w)|(1-|w|^2)^{1+\delta}}{|1-z\overline{w}|^{3+\delta}}dA(w)\\
	\lesssim&\sup_{z \in \D}(1-|z|^2)^\b\left(\l\right)^\gamma \int_{\D\backslash\Omega_{\epsilon}(f)} \f{|f'(w)|(1-|w|^2)^\b\left(\log\frac{e}{1-|w|^2}\right)^\gamma(1-|w|^2)^{1+\delta-\b}}{|1-z\overline{w}|^{3+\delta}\left(\log\frac{e}{1-|w|^2}\right)^\gamma}dA(w)\\
	\lesssim& \epsilon\sup_{z \in \D}(1-|z|^2)^\b\left(\l\right)^\gamma \int_{\D}\f{(1-|w|^2)^{1+\delta-\beta}}{|1-z\overline{w}|^{3+\delta}\left(\log\frac{e}{1-|w|^2}\right)^\gamma}dA(w)\\
	\lesssim& \epsilon.
	\end{align*}
	Hence $g\in \Blog$.
	To complete the proof, it is only necessary to show that $g\in DT_p^q(\a)$.
	Since $g(z)=f_1(z)-f_1(0)$, we obtain
	\begin{equation}\label{2.4}
	\begin{aligned}
	|g'(z)|^p=|f_1'(z)|^p
	\lesssim\left(
	\int_{\Omega_{\epsilon}(f)}\f{|f'(w)|(1-|w|^2)^{1+\delta}}{|1-z\overline{w}|^{3+\delta}}dA(w) \right)^p.
	\end{aligned}	
	\end{equation}
Now, we  divide the remaining proof into two cases.

{\bf Case $1\le p<\a+3$. }   When $p=1$, it is clear that
\begin{align*}
	|g'(z)|
	\lesssim&\|f\|_{\Blog}
	\int_{\Omega_{\epsilon}(f)}\f{(1-|w|^2)^{1+\delta-\b}}{|1-z\overline{w}|^{3+\delta}\left(\log\frac{e}{1-|w|^2} \right)^\gamma}dA(w).
\end{align*}
When $1<p<\a+3$, using H\"older's inequality and Lemma \ref{log},
    \begin{align*}
    &|g'(z)|^p    \lesssim\left(
    \int_{\Omega_{\epsilon}(f)}\f{|f'(w)|^p(1-|w|^2)^{p+\delta}}{|1-z\overline{w}|^{3+\delta}}dA(w) \right)
    \left(\int_{\Omega_{\epsilon}(f)}\f{(1-|w|^2)^{\delta}}{|1-z\overline{w}|^{3+\delta}}dA(w) \right)^{p-1}\\
    \lesssim&\|f\|_{\Blog}^p\left(
    \int_{\Omega_{\epsilon}(f)}\f{(1-|w|^2)^{p+\delta-p\b}}{|1-z\overline{w}|^{3+\delta}\left(\log\frac{e}{1-|w|^2} \right)^{p\gamma}}dA(w) \right)\left(
    \int_{\Omega_{\epsilon}(f)}\f{(1-|w|^2)^{\delta}}{|1-z\overline{w}|^{3+\delta}}dA(w) \right)^{p-1}\\
    \lesssim&\|f\|_{\Blog}^p\left(
    \int_{\Omega_{\epsilon}(f)}\f{(1-|w|^2)^{p+\delta-p\b}}{|1-z\overline{w}|^{3+\delta}\left(\log\frac{e}{1-|w|^2} \right)^{p\gamma}}dA(w) \right)\left(
    \int_{\D}\f{(1-|w|^2)^{\delta}}{|1-z\overline{w}|^{3+\delta}}dA(w) \right)^{p-1}\\
    \lesssim&\frac{\|f\|_{\Blog}^p}{(1-|z|^2)^{p-1}}\left(
    \int_{\Omega_{\epsilon}(f)}\f{(1-|w|^2)^{p+\delta-p\b}}{|1-z\overline{w}|^{3+\delta}\left(\log\frac{e}{1-|w|^2} \right)^{p\gamma}}dA(w) \right).
    \end{align*}
	Then, applying Fubini's theorem, it follows that
	\begin{align*}
	&\|g\|_{DT_p^q(\a)}^q	\asymp \int_{\T}\left( \int_{\Gamma(\eta)}|g'(z)|^p(1-|z|^2)^{\a} dA(z) \right)^{\frac{q}{p}}|d\eta|\\
	\lesssim&\int_{\T}\left( \int_{\Gamma(\eta)}
	\frac{\|f\|_{\Blog}^p(1-|z|^2)^{\a}}{(1-|z|^2)^{p-1}}\left(
	\int_{\Omega_{\epsilon}(f)}\f{(1-|w|^2)^{p+\delta-p\b}}{|1-z\overline{w}|^{3+\delta}\left(\log\frac{e}{1-|w|^2} \right)^{p\gamma}}dA(w) \right)
    dA(z) \right)^{\frac{q}{p}}|d\eta|\\
	\lesssim&\|f\|_{\Blog}^q\int_{\T}\left( \int_{\Omega_{\epsilon}(f)}
	\frac{(1-|w|^2)^{p+\delta-p\beta}}{\left(\log\frac{e}{1-|w|^2} \right)^{p\gamma}}
	\left(
	\int_{\Gamma(\eta)}\f{dA(z)}{(1-|z|^2)^{p-1-\a}|1-z\overline{w}|^{3+\delta}} \right)
	dA(w) \right)^{\frac{q}{p}}|d\eta|.
	\end{align*}
	Note that for $z\in\Gamma(\eta), |1-\ol{\eta}z|\asymp1-|z|^2$. Hence, for any $s>\delta+1$ and $\a+3>p$, using Lemma \ref{le3}, we obtain
	\begin{align*}
	 \int_{\Gamma(\eta)}\f{dA(z)}{(1-|z|^2)^{p-1-\a}|1-z\overline{w}|^{3+\delta}} 
	\lesssim&\int_{\D}\f{(1-|z|^2)^{s}}{|1-z\overline{w}|^{3+\delta}|1-\ol{\eta}z|^{p-1-\a+s}}dA(z)\\
	\lesssim&\frac{1}{|1-\ol{\eta}w|^{\delta+p-\a}}.
	\end{align*}
	Hence, using Lemma \ref{le1}, we get
	\begin{equation} \label{p1}
	\begin{aligned}
	 \|g\|_{DT_p^q(\a)}^q	\lesssim&\|f\|_{\Blog}^q\int_{\T}\left(\int_{\Omega_{\epsilon}(f)}\frac{(1-|w|^2)^{p+\delta-p\beta}}{|1-\ol{\eta}w|^{\delta+p-\a}\left(\log\frac{e}{1-|w|^2} \right)^{p\gamma} }
	dA(w)\right)^{\frac{q}{p}}|d\eta|\\
	=&\|f\|_{\Blog}^q\int_{\T}\left(\int_{\Omega_{\epsilon}(f)}\left( \frac{1-|w|^2}{|1-\ol{\eta}w|}  \right)^{\delta+p-\a}\frac{(1-|w|^2)^{\a-p\b}}{\left(\log\frac{e}{1-|w|^2} \right)^{p\gamma} } dA(w)  \right)^{\frac{q}{p}}|d\eta|	\\
	=&\|f\|_{\Blog}^q\int_{\T}\left(\int_{\D}\left( \frac{1-|w|^2}{|1-\ol{\eta}w|}  \right)^{\delta+p-\a}\frac{(1-|w|^2)^{\a-p\b}\chi_{\Omega_{\epsilon}(f)}(w)}{\left(\log\frac{e}{1-|w|^2} \right)^{p\gamma}} dA(w)
	\right)^{\frac{q}{p}}|d\eta|\\
	\lesssim&\|f\|_{\Blog}^q\int_{\T}\left(\int_{\Gamma(\eta)}\frac{(1-|w|^2)^{\a-p\b}\chi_{\Omega_{\epsilon}(f)}(w)}{\left(\log\frac{e}{1-|w|^2} \right)^{p\gamma}} dA(w) \right)^{\frac{q}{p}}|d\eta|\\
	\lesssim&\|f\|_{\Blog}^q\int_{\T}\left(\int_{\gaon}\frac{(1-|w|^2)^{\a-p\b}}{\left(\log\frac{e}{1-|w|^2} \right)^{p\gamma}}dA(w) \right)^{\frac{q}{p}}|d\eta|\\
	<&\infty.
	\end{aligned}
\end{equation}
Hence, $g\in DT_p^q(\a)$. Therefore, for any $\epsilon>0$, there exists a function $g\in DT_p^q(\a)\cap \Blog $ such that $\|f-g\|_{\Blog}\lesssim \epsilon$, i.e., $f\in \dt$. 

    {\bf Case $0<p<1$.} 
    Since the function $|f'(z)|(1-|z|^2)^\beta\left(\l\right) ^{\gamma}$ is uniformly continuous with respect to the Bergman metric on $\D$, there exists $\rho\in(0,1)$ such that for any $z,w\in\D$ with $\b(z,w)<3\rho$,
    $$
    \left| |f'(z)|(1-|z|^2)^\b\left(\l\right) ^{\gamma}-|f'(w)|(1-|w|^2)^\b \left(\log\frac{e}{1-|w|^2}\right) ^\gamma \right|<\frac{\epsilon}{2}.
    $$ 
    Let $\{a_j\}$ be an $(r,\kappa)$ lattice.
    Let $
    \mathcal{M}=\left\{j:D(a_j,\rho)\cap\on\neq\emptyset\right\}.$   It is obvious that
    $
    \bigcup_{j\in\mathcal{M}}D(a_j,2\rho)\subset\Omega_{\frac{\epsilon}{2}}(f).
    $
    Hence, using $(\ref{2.4})$ and subharmonic property of $|f'|$, it follows that
    \begin{align*}
    |g'(z)|^p
    &\lesssim\left(
    \int_{\Omega_{\epsilon}(f)}\f{|f'(w)|(1-|w|^2)^{1+\delta}}{|1-z\overline{w}|^{3+\delta}}dA(w) \right)^p\\
    &\le\left(\sum_{j\in\mathcal{M}}
    \int_{D(a_j,\rho)}\f{|f'(w)|(1-|w|^2)^{1+\delta}}{|1-z\overline{w}|^{3+\delta}}dA(w) \right)^p\\
    &\lesssim\sum_{j\in\mathcal{M}}\frac{(1-|a_j|^2)^{p+p\delta}}{|1-\ol{a_j}z|^{p(3+\delta)}}\left( \int_{D(a_j,\rho)}|f'(w)|dA(w)  \right)^p \\
    &\lesssim\sum_{j\in\mathcal{M}}\frac{(1-|a_j|^2)^{p+p\delta+2p-2}}{|1-\ol{a_j}z|^{p(3+\delta)}} \int_{D(a_j,2\rho)}|f'(w)|^pdA(w)  \\
    &\lesssim\sum_{j\in\mathcal{M}}\int_{D(a_j,2\rho)}\frac{|f'(w)|^p(1-|w|^2)^{p+p\delta+2p-2}}{|1-\ol{w}z|^{p(3+\delta)}}dA(w)\\
    &\lesssim\|f\|^p_{\Blog}\int_{\Omega_{\frac{\epsilon}{2}}(f)}\frac{(1-|w|^2)^{p+p\delta+2p-2-p\b}}{|1-\ol{w}z|^{p(3+\delta)}\left( \log\frac{e}{1-|w|^2}\right)^{p\gamma}}dA(w).
    \end{align*}	
Applying Fubini's theorem, we get
    \begin{align*}
  & \|g\|^q_{DT_p^q(\a)} 
    \asymp\int_{\T}\left( \int_{\Gamma(\eta)}|g'(z)|^p(1-|z|^2)^{\a} dA(z) \right)^{\frac{q}{p}}|d\eta|\\
    \lesssim&\int_{\T}\left( \int_{\Gamma(\eta)} \|f\|^p_{\Blog}\int_{\Omega_{\frac{\epsilon}{2}}(f)}\frac{(1-|w|^2)^{p+p\delta+2p-2-p\b}}{|1-\ol{w}z|^{p(3+\delta)}\left( \log\frac{e}{1-|w|^2}\right)^{p\gamma}}dA(w) (1-|z|^2)^{\a} dA(z) \right)^{\frac{q}{p}}|d\eta|\\
    \lesssim&\|f\|_{\Blog}^q\int_{\T}\left(   \int_{\Omega_{\frac{\epsilon}{2}}(f)}\frac{(1-|w|^2)^{p+p\delta+2p-2-p\b}}{\left( \log\frac{e}{1-|w|^2}\right)^{p\gamma}} \left(\int_{\Gamma(\eta)} \frac{(1-|z|^2)^{\a}}{|1-\ol{w}z|^{p(3+\delta)}} dA(z)\right) dA(w) \right)^{\frac{q}{p}}|d\eta|.
    \end{align*}
Notice that for $z\in\Gamma(\eta)$, $|1-\ol{\eta}z|\asymp1-|z|^2$. Hence, for any $t>\max\left\{\a+1,\a+\frac{p}{q}\right\}$ and $p(3+\delta)>t+2$, using Lemma \ref{le3}, we get
\begin{align*}
\int_{\Gamma(\eta)} \frac{(1-|z|^2)^{\a}}{|1-\ol{w}z|^{p(3+\delta)}} dA(z)
\lesssim&\int_{\D}\frac{(1-|z|^2)^t}{|1-\ol{\eta}z|^{t-\a}|1-\ol{w}z|^{p(3+\delta)}}dA(z)\\
\lesssim&\frac{1}{(1-|w|^2)^{p(3+\delta)-t-2}|1-\ol{\eta}w|^{t-\a}}.
\end{align*}
Since  $(\ref{11})$ still holds if $\epsilon$ is replaced by $\frac{\epsilon}{2}$, using Lemma \ref{le1} and similar to $(\ref{p1})$, we have
\begin{align*}
&\|g\|_{DT_p^q(\a)}^q\\
\lesssim&\|f\|_{\Blog}^q\int_{\T}\left(   \int_{\Omega_{\frac{\epsilon}{2}}(f)}\frac{(1-|w|^2)^{p+p\delta+2p-2-p\b}}{  (1-|w|^2)^{p(3+\delta)-t-2}|1-\ol{\eta}w|^{t-\a}  \left( \log\frac{e}{1-|w|^2}\right)^{p\gamma}}  dA(w) \right)^{\frac{q}{p}}|d\eta|\\
\lesssim&\|f\|_{\Blog}^q\int_{\T}\left(   \int_{\Omega_{\frac{\epsilon}{2}}(f)}
\left( \frac{1-|w|^2}{|1-\ol{\eta}w|} \right)^{t-\a}\frac{(1-|w|^2)^{\a-p\b}}{\left(\log\frac{e}{1-|w|^2} \right)^{p\gamma}}dA(w) \right)^{\frac{q}{p}}|d\eta|\\
\lesssim&\|f\|_{\Blog}^q\int_{\T}\left(\int_{\Gamma(\eta)\cap\Omega_{\frac{\epsilon}{2}}(f)}\frac{(1-|w|^2)^{\a-p\b}}{\left(\log\frac{e}{1-|w|^2} \right)^{p\gamma}}dA(w) \right)^{\frac{q}{p}}|d\eta|\\
<&\infty.
\end{align*}
That is, $g\in DT_p^q(\a)$. Therefore, for any $\epsilon>0$, there exists a function $g\in DT_p^q(\a)\cap \Blog $ such that $\|f-g\|_{\Blog}\lesssim \epsilon$, i.e., $f\in \dt$. 
The proof is complete.
\end{proof}

The characterization in Theorem \ref{th1} not only encompasses many existing results but also presents a new characterization  in contrast to the findings in the original paper (see \cite{gmp, ql,blz,gg}). This shows the generality of our approach. In particular, when $\b=1, \gamma=0$, we get a new characterization of Bloch functions in \(C_\mathcal{B}(\mathcal{D}^p_{\alpha} \cap \mathcal{B})\) (see \cite{gg}).
\begin{Corollary}
Let $1\le p<\infty$, $p-2<\a\le p-1$ and $f\in\mathcal{B}$. Then \(f\in  C_\mathcal{B}(\mathcal{D}^p_{\alpha} \cap \mathcal{B})\) if and only if for any $\epsilon>0$,
\begin{align*}
	\int_{\T}\left(\int_{ \Gamma(\eta)\cap \widetilde{\Omega}_{\epsilon}(f) }(1-|z|^2)^{\a-1-p} dA(z)
	\right)|d\eta|<\infty.
\end{align*}
Here $$
\widetilde{\Omega}_{\epsilon}(f)=\left\{z\in \D: |f'(z)|(1-|z|^2) \geq \epsilon\right\}.
$$
\end{Corollary}

For another case, when $\b=1$ and $\gamma=1$, we obtain a new characterization of logarithmic Bloch functions in
\(C_{\mathcal{B}_{{\log}}}(\mathcal{D}^2_{\alpha}\cap\mathcal{B}_{{\log}})\) (see \cite{blz,ql}).

\begin{Corollary}
Let $0<\a\le 1$ and $f\in\mathcal{B}_{\log}$. Then \(f\in C_{\mathcal{B}_{{\log}}}(\mathcal{D}^2_{\alpha}\cap\mathcal{B}_{{\log}})\) if and only if for any $\epsilon>0$,
\begin{align*}
	\int_{\T}\left(\int_{ \Gamma(\eta)\cap \widehat{\Omega}_{\epsilon}(f)}\frac{(1-|z|^2)^{\a-3}}{\left(\log\frac{e}{1-|z|^2}\right)^{2}} dA(z)
	\right)|d\eta|<\infty.
\end{align*}
Here
$$
\widehat{\Omega}_{\epsilon}(f)=\left\{z\in \D: |f'(z)|(1-|z|^2) \log\frac{e}{1-|z|^2}   \geq \epsilon\right\}.
$$
\end{Corollary}\msk

\noindent \textbf{Acknowledgments.} 
 The   authors are supported  by GuangDong Basic and Applied Basic Research Foundation (No. 2023A1515010614).\msk

\noindent\textbf{Data Availability.}  Data sharing is not applicable for this article as no datasets were generated or analyzed during the current study.\msk

 \noindent \textbf{Conflict of interest.}   The authors declare no competing interests.\msk

\end{document}